\documentclass[12pt,a4paper,reqno]{amsart}

\input xy
\xyoption{all}

\DeclareMathAlphabet{\mathpzc}{OT1}{pzc}{m}{it}

\usepackage{amsfonts,amsmath,amssymb,indentfirst,mathrsfs,amscd}
\usepackage[mathscr]{eucal}
\usepackage[active]{srcltx}
\usepackage{graphicx}

\usepackage{enumitem}

\usepackage[dvips]{color}
\usepackage{color}

\author[I. Biswas]{Indranil Biswas}
\address{School of Mathematics, Tata Institute of Fundamental
Research, Homi Bhabha Road, Bombay 400005, India}
\email{indranil@math.tifr.res.in}

\author[T. L. G\'omez]{Tom\'as L. G\'omez}
\address{Instituto de Ciencias Matem\'aticas (CSIC-UAM-UC3M-UCM),
C/ Nicolas Cabrera 15, 28049 Madrid, Spain}
\email{tomas.gomez@icmat.es}

\author[M. Logares]{Marina Logares}
\address{Instituto de Ciencias Matem\'aticas (CSIC-UAM-UC3M-UCM),
C/ Nicolas Cabrera 15, 28049 Madrid, Spain}
\email{marina.logares@icmat.es}

\title[Integrable systems and Torelli Theorems]{Integrable systems and Torelli
  theorems for the moduli spaces of parabolic bundles and parabolic Higgs bundles}

\thanks{The authors want to thank the support by MINECO: ICMAT Severo
  Ochoa project SEV-2011-0087 and MTM2010/17389. 
We acknowledge the support of the grant 612534 MODULI within the 7th
European Union Framework Programme.
 The second and third author also wants to thank Tata Institute for Fundamental Research (Mumbai) 
where this work was finished. The third author was
  also partially supported by FCT (Portugal) with European Regional
  Development Fund (COMPETE), national funds through the projects
  PTDC/MAT/098770/2008 and PTDC/MAT/099275/2008. The first author is supported by  J. C. Bose
Fellowship.} 

\subjclass[2000]{Primary: 14D22 Secondary: 14D20}

\keywords{Parabolic bundles, Higgs field, Torelli theorem}

\DeclareMathOperator{\rk}{rk \,}

\DeclareMathOperator{\pdeg}{pardeg \,}
\DeclareMathOperator{\pmu}{par\mu \,}
\DeclareMathOperator{\img}{im \,}
\DeclareMathOperator{\coker}{coker\,}
\DeclareMathOperator{\PH}{ParHom \,}
\DeclareMathOperator{\SPH}{SParHom\,}

\DeclareMathOperator{\PE}{ParEnd \,}
\DeclareMathOperator{\SPE}{SParEnd\,}
\DeclareMathOperator{\ad}{ad\,}

\DeclareMathOperator{\Id}{Id\,}

\DeclareMathOperator{\U}{U}

\DeclareMathOperator{\codim}{codim}

\DeclareMathOperator{\Spec}{Spec}

\newcommand{\wt}{\widetilde}

\begin{document}

\newtheorem{thm}{Theorem}[section]
\newtheorem*{nonumthm}{Theorem}
\newtheorem{prop}[thm]{Proposition}
\newtheorem{lem}[thm]{Lemma}
\newtheorem{cor}[thm]{Corollary}
\newtheorem{conjecture}{Conjecture}

\theoremstyle{definition}
\newtheorem{defn}[thm]{Definition}
\newtheorem{ex}[thm]{Example}
\newtheorem{as}{Assumption}

\theoremstyle{remark}
\newtheorem{rmk}[thm]{Remark}

\theoremstyle{remark}
\newtheorem*{prf}{Proof}

\newcommand{\iacute}{\'{\i}}
\newcommand{\norm}[1]{\lVert#1\rVert}

\newcommand{\lto}{\longrightarrow}
\newcommand{\hra}{\hookrightarrow}

\newcommand{\suchthat}{\;\;|\;\;}
\newcommand{\dbar}{\overline{\partial}}

\newcommand{\cC}{\mathcal{C}}
\newcommand{\cD}{\mathcal{D}}
\newcommand{\cE}{\mathcal{E}}
\newcommand{\cG}{\mathcal{G}}
\newcommand{\cH}{\mathcal{H}}
\newcommand{\cF}{\mathcal{F}}
\newcommand{\cO}{\mathcal{O}}
\newcommand{\cL}{\mathcal{L}}
\newcommand{\cM}{\mathcal{M}}
\newcommand{\cN}{\mathcal{N}}
\newcommand{\cP}{\mathcal{P}}
\newcommand{\cS}{\mathcal{S}}
\newcommand{\cU}{\mathcal{U}}
\newcommand{\cX}{\mathcal{X}}
\newcommand{\cT}{\mathcal{T}}
\newcommand{\cV}{\mathcal{V}}
\newcommand{\cB}{\mathcal{B}}
\newcommand{\cR}{\mathcal{R}}
\newcommand{\cJ}{\mathcal{J}}

\newcommand{\ext}{\mathrm{ext}}
\newcommand{\x}{\times}

\newcommand{\mM}{\mathscr{M}}

\newcommand{\CC}{\mathbb{C}}
\newcommand{\QQ}{\mathbb{Q}}
\newcommand{\PP}{\mathbb{P}}
\newcommand{\HH}{\mathbb{H}}
\newcommand{\RR}{\mathbb{R}}
\newcommand{\ZZ}{\mathbb{Z}}
\newcommand{\EE}{\mathbb{E}}

\renewcommand{\lg}{\mathfrak{g}}
\newcommand{\lh}{\mathfrak{h}}
\newcommand{\lu}{\mathfrak{u}}
\newcommand{\la}{\mathfrak{a}}
\newcommand{\lb}{\mathfrak{b}}
\newcommand{\lm}{\mathfrak{m}}
\newcommand{\lgl}{\mathfrak{gl}}
\newcommand{\too}{\longrightarrow}

\newcommand{\imat}{i}

\hyphenation{mul-ti-pli-ci-ty}

\hyphenation{mo-du-li}

\begin{abstract}
  We prove a Torelli theorem for the moduli space of semistable
  parabolic Higgs bundles over a smooth complex projective algebraic
  curve under the assumption that the parabolic weight system is
  generic. When the genus is at least two, using this result we also
  prove a Torelli theorem for the moduli space of semistable parabolic
  bundles of rank at least two with generic parabolic weights. The key
input in the proofs is a method of \cite{Hu}.
\end{abstract}

\maketitle

\section{Introduction}

The classical theorem by R. Torelli \cite{crs} says that a smooth
complex algebraic curve is determined by the isomorphism class of its
polarized Jacobian up to isomorphism. Similar theorems in many
contexts have been worked out, e.g., for moduli spaces of stable
vector bundles \cite{T,NR,MN} and moduli spaces of stable Higgs
bundles \cite{BG}.  As far as moduli spaces of parabolic or parabolic
Higgs bundles with fixed determinant (see definition below) are concerned, 
a set of Torelli theorems were proved \cite{BBB,BHK,sebastian,GL}.
Here we deal with the non-fixed determinant situation.

In \cite{Hu}, Hurtubise investigated algebraically completely integrable
systems satisfying certain conditions. His main result is to extract an
algebraic surface out of an integrable system. We observe that a
moduli space of parabolic Higgs bundles is an example of the model
of completely integrable systems studied in \cite{Hu}.

The above mentioned assumption that the determinant is not fixed stems from
the fact that in the set-up of \cite{Hu} the Lagrangians in the fibers
are required to be Jacobians, while fixing the determinant amounts
to making the fibers Prym varieties. To consider the moduli
spaces with fixed determinant with our techniques, we would need an
analogue of our main tool, namely Theorem 1.11 of \cite{Hu}, but for an integrable
system in which the fibers are Prym
varieties instead of Jacobians. This is planned for future work.

We will prove the following theorems.

\begin{thm}[Main Theorem]\label{theoremPH}
Let $X$ and $X'$ be smooth projective curves with genus $g$ and
parabolic points $D$ and $D'$ respectively. Let
$\cM_{X}(d,r,\alpha)$ (respectively, $\cM_{X'}(d,r,\alpha)$) be the
moduli space of stable parabolic Higgs bundles over $X$
(respectively, $X'$) endowed with the usual~ $\mathbb{C}^*$ action
(cf. \eqref{action}) and the determinant line bundle $\cL$
(respectively, $\cL'$) 
(cf. \eqref{linebundle}). 
If there is a
$\mathbb{C}^*$-equivariant isomorphism between $\cM_X(d,r,\alpha)$
and $\cM_{X'}(d,r,\alpha)$, such the the pullback of the 
N\'eron-Severi class $NS(\cL')$
is $NS(\cL)$, then there exists an isomorphism between
$X$ and $X'$ inducing a bijection between the parabolic points $D$
and $D'$, whenever the following conditions on the genus and the
rank are satisfied,
\begin{itemize}
\item if $g=2$ then $r\ge 5$,
\item if $g=3$ then $r\ge 3$,
\item if $g\ge 4$ then $r\ge 2$.
\end{itemize}
\end{thm}

Since the moduli space of stable parabolic bundles sits inside the
moduli space of stable parabolic Higgs bundles, in all cases where its
codimension is greater than two, we get the following extension of the
Torelli theorem for the moduli space of stable parabolic bundles given
in \cite{BBB}.

\begin{thm}\label{theoremP}
Let $X$ and $X'$ be smooth projective curves with genus $g$
and parabolic points $D$ and $D'$ respectively. 
Let $M_X(d,r,\alpha)$ be the moduli space of stable 
parabolic bundles over $X$ (respectively, $M_{X'}(d,r,\alpha)$),
and let $\cL$ (respectively, $\cL'$) 
be the determinant line bundle (cf. \eqref{linebundle}). 
If there is an isomorphism between $M_X(d,r,\alpha)$ and 
$M_{X'}(d,r,\alpha)$ such that the pullback of $NS(\cL')$
is $NS(\cL)$, then
there exists an isomorphism between $X$ and $X'$ inducing a bijection
between the 
parabolic points $D$ and $D'$, whenever the following conditions 
on the genus and the rank are satisfied,
\begin{itemize}
\item if $g=2$ then $r\ge 5$,
\item if $g=3$ then $r\ge 3$,
\item if $g\ge 4$ then $r\ge 2$.
\end{itemize}
\end{thm}

\section{Preliminaries}

Let $X$ be an irreducible smooth projective algebraic curve over $\mathbb C$. The
holomorphic cotangent bundle of $X$ will be denoted by $K$. Let 
$\{p_1\, ,\cdots\, , p_n\}$ be a set of distinct 
\emph{parabolic points} in $X$ and let $D
\,=\, p_1+\ldots +p_n$ the corresponding reduced effective divisor.
A \emph{parabolic bundle} on $X$ with parabolic structure over
$D$ consists of a holomorphic vector bundle $E$ equipped with a weighted
flag over each \emph{parabolic point} $p\in D$, that is a filtration of subspaces
$$
E\vert_p\,=\,E_{p,1}\,\supset\, \cdots \,\supset\, E_{p,r(p)}\,\supset\, 
E_{p,r(p)+1}\,=\, 0
$$
together with a system of \emph{parabolic weights}
$$
0\,\leq\, \alpha_{1}(p)\,<\,\cdots \,<\, \alpha_{r(p)}(p)\,<\,1\, .
$$
The \emph{parabolic degree} and \emph{parabolic slope} of $E$ are defined as follows:
$$
\pdeg(E) \,:=\,\deg (E) + \sum_{p \in D}\sum_{i=1}^{r(p)} \alpha_i(p)\cdot m_i (p)
\quad\quad \pmu (E) \,:=\, \frac{\pdeg(E)}{\rk (E)}\, ,
$$
where $m_i (p)\,:=\, \dim (E_{p,i} /E_{p,i+1})$ is the multiplicity of the
parabolic weight $ \alpha_i(p)$.
The parabolic bundle is called \emph{stable} (respectively, \emph{semistable}) if
for all subbundles $0\,\not=\,V\, \subsetneq\, E$,
\begin{equation}
  \label{stab}
\pmu (V)\,<\, \pmu (E) \qquad \text{(respectively, $\pmu (V)\,\leq\, \pmu (E)$)}
\end{equation}
where $V$ has the induced parabolic structure.
Given rank and degree the system of parabolic weights is called \emph{generic} if every
semistable parabolic bundle is stable. We note that the semistability condition describes 
hyperplanes (or \emph{walls}) in the space of weights. Hence the genericity condition
means that the parabolic weights lie in the interior of the chambers defined by the walls.

We denote by $M_{X}(d,r,\alpha)$ the moduli space of stable 
parabolic bundles over
$X$ with degree $d$, rank $r$ and generic weights $\alpha$.
This moduli space is a smooth projective variety with
$$
\dim M_{X}(d,r,\alpha) \,= \,
r^2 (g-1) +1 + \frac{1}{2}\sum_{p\in D}\sum_{i=1}^{r(p)} (r^2- m_{i}(p)^{2})\, .
$$
For notational convenience we assume that the flag is  
\emph{full} that is, $m_i(p)\,=\,1$ for all $p$ and $i$,
so $r(p)\,=\,r$ for all $p$, but all the results generalize to non full flags case. Henceforth, we will only consider full flags. 
Therefore, 
$$
\dim M_{X}(d,r,\alpha)\,=\,r^2(g-1)+1+ \frac{1}{2} nr  (r-1) \; .
$$

An endomorphism of a parabolic bundle $E$ is called \emph{non-strongly
parabolic} if, for all $p\in D$ and $i$, 
$$
\varphi (E_{p,i})\,\subset\, E_{x,i},
$$
and it is called \emph{strongly parabolic} if
$$
\varphi (E_{p,i})\,\subset\, E_{p,i+1}\, .
$$
The sheaves of non-strongly and strongly parabolic endomorphisms are
denoted by $\PE(E)$ and $\SPE(E)$ respectively.

A \emph{parabolic Higgs bundle} is a pair $(E,\Phi)$ where $E$ is a
parabolic bundle and
$$
\Phi\,:\,E\,\too\, E\otimes K(D)\,=\, E\otimes K\otimes {\mathcal O}_X(D)
$$
is a \emph{strongly parabolic} 
homomorphism, i.e.,
$$
\Phi(E_{x,i}) \,\subset\, E_{x,i+1}\otimes K(D)_x
$$ 
for each point $x\,\in \,D$ and all $i$.
A parabolic Higgs bundle is \emph{stable} (respectively, \emph{semistable}) if
the inequality \eqref{stab} is satisfied for those $V$ with $\Phi(V)\, \subset\,
V\otimes K(D)$.

Let $\cM_{X}(d,r,\alpha)$ denote the moduli space of stable parabolic Higgs
bundles with degree $d$, rank $r$ and generic weights $\alpha$.
It is a smooth quasiprojective variety that satisfy
$$
\dim \cM_{X}(d,r,\alpha)\,=\, 2r^2 (g-1) +2 + nr(r-1)
\,=\, 2\cdot  \dim M_{X}(d,r,\alpha)
$$
(recall that the quasiparabolic flags are full).

For any $E\in M_{X}(r,d,\alpha)$ the tangent space at $E$, $T_{E}M_{X}(r,d,\alpha)$,  is $H^{1}(\PE(E))$. 
Also, the parabolic version of Serre duality gives an isomorphism
$$
H^1(\PE(E))^\ast \,\cong\, H^0 (\SPE(E) \otimes K(D)) \; .
$$
Therefore,  the total space of the cotangent 
bundle $T^\ast M_{X}(r,d,\alpha)$ is a
Zariski open subset of $\cM_{X}(r,d,\alpha)$.

The moduli space of parabolic Higgs bundles is 
endowed with a $\mathbb{C}^*$
action, where $t\in \mathbb{C}^*$ acts as scalar multiplication
on the Higgs field
\begin{equation}
\label{action}
(E,\Phi) \longmapsto (E,t\cdot \Phi)
\end{equation}
The total space of the cotangent bundle $T^\ast M_{X}(r,d,\alpha)$ 
also has a canonical $\mathbb{C}^*$ action given by scalar
multiplication on the fibers. Both actions are compatible, in the
sense that the inclusion of the cotangent in the moduli space of
Higgs bundles is $\mathbb{C}^*$ equivariant.

\section{The Hitchin system}

Let $\mathcal{K(D)}$ denote the total space of the line bundle $K(D)$ over $X$, and
let $\gamma\,:\,\mathcal{K(D)}\,\too\, X$ be the natural projection. Let
$$
\widetilde{x}\,\in\, H^0(\mathcal{K(D)},\, \gamma^* K(D))
$$
be the tautological section whose evaluation at any point $z$ is $z$ itself. The
characteristic polynomial of a Higgs field $\Phi$ is
\begin{equation}\label{eq:char_pol}
\det(\widetilde{x}\cdot\Id - \gamma^* \Phi) \,=\, \widetilde{x}^r + \wt{s}_1 \widetilde{x}^{r-1} +
\wt{s}_2 \widetilde{x}^{r-2} +\cdots +\wt{s}_r.
\end{equation}
The sections $\widetilde{s}_{i}$, descent to $X$, meaning there are sections $s_i\,\in \,H^0(X,\,K^i(iD))$ such that $\wt{s}_i\,=\,\gamma^* s_i$.
Since $\Phi$ is strongly parabolic its residue at each parabolic
point is nilpotent, and hence $s_i\in H^0(X,\,K^i((i-1)D))$.
Therefore, there is a morphism, called the \emph{Hitchin map},
\begin{equation}\label{eqn:hitchin}
H\,:\, \cM_{X}(d,r,\alpha)\,\longrightarrow\, 
\cU\,:=\,\bigoplus_{i=1}^{r} H^{0}(X,K^{i}((i-1)D))\, .
\end{equation}
This morphism is proper \cite{Hi}, and it induces an isomorphism on globally defined
algebraic functions, i.e., the lower arrow in the following commutative diagram is an 
isomorphism
\begin{equation}
\label{eq:globalsections}
\xymatrix{
{\cM_{X}(d,r,\alpha)} \ar[d]^{a} \ar[r]^-{H} & {\cU} \ar@{=}[d]\\
\Spec\Gamma(\cM_{X}(d,r,\alpha)) \ar[r]^-{\cong}&  \Spec\Gamma(\cU) 
}
\end{equation}
The variety $\cM_{X}(d,r,\alpha)$ has a natural holomorphic symplectic structure,
and the Hitchin map defines an algebraically complete integrable system, in particular,
the fibers of $H$ are Lagrangians (these is explained in \cite{GL}).

When the parabolic set is empty ($n\,=\, 0$), Hausel proved that the nilpotent 
cone $H^{-1}(0)$ coincides with the downwards Morse flow on $\cM_{X}(d,r,\alpha)$
giving a deformation retraction of $\cM_{X}(d,r,\alpha)$ to $H^{-1}(0)$
\cite[Theorem 5.2]{hausel}. The proof in \cite{hausel} can be translated 
into the parabolic situation word by word.

The fiber of $H$ over a point $u \,\in\, \cU$ is canonically isomorphic
to the Jacobian of a curve called the \emph{spectral curve}; we now recall its
construction. 

Given a point $u \,=\, (s_1\, ,\cdots\, ,s_r)\,\in\, \cU$ consider the curve
$X_u\,\subset\,\mathcal{K(D)}$ defined by the equation
$$
\widetilde{x}^r + s_1 \widetilde{x}^{r-1} + s_2 \widetilde{x}^{r-2} +\cdots +s_r\,=\,0
$$
(compare it with (\ref{eq:char_pol})). Note that when $X_u$ is reduced, the projection
$$
\rho\,:=\, \gamma\vert_{X_u}\,:\, X_u\,\longrightarrow\, X
$$
is a ramified covering of  $X$ of degree $r$ which is
completely ramified over the parabolic points. Denote by $R_u$ the ramification
divisor on $X_u$. Denote by $\cS$ the family of spectral curves over $\cU$.

\begin{prop}
For any $u\,\in\,\cU$ such that the corresponding spectral curve
$X_u$ is smooth, the fiber $H^{-1}(u)$ is identified with 
${\rm Pic}^{d+r(r-1)(2g-2+n)/2}(X_u)$.
\end{prop}

\begin{proof}
It follows from the proof of Proposition 3.6. in \cite{BNR}.
\end{proof}

Let $\cE$ be a universal bundle on 
$X\times \cM_{X}(d,r,\alpha)$ and
let $q$ be the projection to $\cM_{X}(d,r,\alpha)$.
Fix a point $x\in X$ of the curve. Let $\chi=\chi(E)$
(since we have fixed the rank and degree, this does
not depend on the particular $E$ chosen, and can be
calculated by the Riemann-Roch formula). There is a line bundle 
$\cL^x$ defined as follows \cite{KM}
\begin{equation}
\label{linebundle}
\cL^x= {\rm det}(Rq_* \cE)^{-r} \otimes 
(\wedge^{r} \cE|_{x\times \cM})^{\chi}
\end{equation}
where the presence of the second factor is a normalization
that guarantees that
this does not depend on the choice of universal bundle.
Note that this determinant line bundle can also be
defined for the moduli space $M_{X}(d,r,\alpha)$ without
Higgs bundle.

We remark that this line bundle is invariant under the
standard $\mathbb{C}^*$ action \eqref{action}
and we can choose a lift of this action.

The fiber of this line bundle over a point corresponding
to a Higgs bundle $(E,\Phi)$ is canonically isomorphic
to
$$
\Big[
(\wedge^{top} H^0(X,E))^* \otimes
(\wedge^{top} H^1(X,E))
\Big]^{\otimes r}
\otimes 
({\wedge} E_x)^{\chi}
$$
Since the curve $X$ is connected, the N\'eron-Severi class 
$NS(\cL^x)$ of the line bundle does not depend on the choice of the
point $x\in X$.

\begin{lem}
Let $u\in \mathcal{U}$ is a point in the Hitchin space corresponding
to a smooth curve, then the restriction of the line bundle
$\cL^x$ to the 
fiber $H^{-1}(u)={\rm Pic}^{d+r(r-1)(2g-2+n)/2}(X_u)$ is a multiple of the principal polarization
of the Jacobian $J(X_u)$ of the spectral curve $X_u$ 
\end{lem}

\begin{proof}
Let $(E,\Phi)$ be a point in the moduli space $\cM$. If it is in
the fiber $H^{-1}(u)$, then there is a line bundle $\eta$ on the
spectral curve $\pi:X_u\longrightarrow X$ such that $E=\pi_* \eta$.
Then the fiber of $\cL^x$ over this point is canonically isomorphic
to
$$
\Big[
(\wedge^{top} H^0(X,\pi_* \eta))^* \otimes
(\wedge^{top} H^1(X,\pi_* \eta))
\Big]^{\otimes r}
\otimes 
({\wedge} (\pi_* \eta)_x)^{\chi}
=$$
$$
\Big[
(\wedge^{top} H^0(X_s, \eta))^* \otimes
(\wedge^{top} H^1(X_s, \eta))
\Big]^{\otimes r}
\otimes 
({\wedge} \eta_{\pi^{-1}(x)})^{\chi}
$$
This is the fiber of a line bundle 
defining a multiple of a principal polarization
of the Jacobian. The last factor is just a normalization, and
the N\'eron-Severi class of the line bundle does not depend on
the choice of the point.
\end{proof}

In \cite{Hu}, Hurtubise considers  (local) integrable systems
$$
\mathbb{H}\,:\,\mathbb{J}\,\longrightarrow\, \mathbb{U}\, ,
$$
where $\mathbb{U}$ is an open subset of $\CC^m$ and $\mathbb{J}$ is a $2m$-dimensional
symplectic variety with holomorphic symplectic form $\Omega$, such that the fibers of
$\mathbb{H}$ are Lagrangian. Furthermore, suppose there is a family of curves
$$
\mathbb{H}'\,:\, \mathbb{S}\,\longrightarrow\, \mathbb{U}
$$
such that for each $u\,\in\, \mathbb{U}$, the fiber $J_u\,=\,\mathbb{H}^{-1}(u)$ is
isomorphic to the Jacobian of $S_u\,=\,\mathbb{H}'{}^{-1}(u)$.
To define the Abel map
$$
I\,:\,\mathbb{S}\,\longrightarrow\, \mathbb{J}
$$
we need a section of $\mathbb{H}'$. This can be done locally on $\mathbb{U}$.
Under the assumption that 
\begin{equation}\label{i20}
I^*\Omega \wedge I^*\Omega\,=\,0
\end{equation}
Hurtubise proves that for the embedding $I$ the variety $\mathbb{S}$
is coisotropic, and the quotienting of $\mathbb{S}$ 
by the null foliation results 
a surface $Q$. The form $I^*\Omega$ descends to $Q$,
and the descended form on $Q$, which we will denote by $\omega$,
is a holomorphic symplectic form \cite[Theorem
1.11]{Hu}. He also proves that, choosing a different Abel map
$I'$ with $I'{}^*\Omega \wedge I'{}^*\Omega\,=\,0$, we have
$I^*\Omega\,=\,I'{}^*\Omega$ when $m\,\geq\, 3$, 
so that the surface $Q$ depends only
on $\mathbb{S}$ and it is independent of the Abel map.
We summarize:

\begin{thm}[{\cite[Theorem 1.11 (i) and (ii)]{Hu}}]\label{thm:hurtubise}
For an integrable system $$\HH\,:\,\mathbb{J}\,\longrightarrow\, \mathbb{U}
\,\subset\, \CC^{m}\, ,$$ with maps $\HH'\,: \,\mathbb{S}\,\longrightarrow\, \mathbb{U}$,
and $I\,:\,\mathbb{S}\,\longrightarrow\, \mathbb{J}$, as described above, there is an
invariant surface $Q$ which only depends on $\mathbb{S}$ and not on the Abel map $I$,
whenever $m\,\ge\, 3$. 
\end{thm}

In \cite[Example 4.3]{Hu} he shows that all these conditions are
satisfied for the usual moduli space of Higgs bundles (i.e., no parabolic 
points), but restricted
to the open subset $U$ of the Hitchin space $\mathcal{U}$ corresponding
to smooth spectral curves
$$
\mathbb{H}: \cM_X|_U \longrightarrow U \; .
$$
Let $q$ be the projection
$q:\mathbb{S}\longrightarrow \mathcal{K}$ sending each point on a
spectral curve to the total space of the cotangent bundle and
let $\omega$ the natural symplectic form on the cotangent.
Hurtubise shows that
$$
I^* \Omega = q^* \omega
$$
It follows that the surface $Q$ is $\mathcal{K}$.

The conditions of the theorem also hold
for the moduli space of strongly parabolic Higgs bundles 
equipped with the Hitchin map, and in this case the surface $Q$ is the
image of $\mathbb{S}\longrightarrow \mathcal{K(D)}$. Note that all spectral curves
go through zero on the fibers over the parabolic points, because the
eigenvalues of the residues are zero. 
Therefore, we obtain the following
Corollary, which will be our main tool in the proof of the Main Theorem.

Note that the integer $m$ in the statement of 
Theorem \ref{thm:hurtubise} is the genus of the spectral curve, which
is equal to $\dim M_X(d,r,\alpha)$ and hence, under the assumptions on genus
and rank of Theorems \ref{theoremPH} and \ref{theoremP} we always
have $m\geq 3$ and hence can apply the Theorem of Hurtubise.

\begin{cor}
\label{MainCor}
Let
$$
\mathbb{H}: \cM_{X}(d,r,\alpha)|_U \longrightarrow U \; .
$$
be the restriction of the Hitchin map on the moduli space of parabolic
Higgs bundles with generic weights $\alpha$ to the open set $U$
corresponding to nonsingular curves (cf. Lemma \eqref{lem:U}).  Then
this integrable system satisfies the conditions of the Theorem of
Hurtubise and the surface $Q$ is the image of $\mathcal{K}$ in
$\mathcal{K(D)}$ under the injective morphism of sheaves
$K\longrightarrow K(D)$.
\end{cor}

\section{Proof of the Theorems}

Let
$$
h\,:\,T^{\ast}M_{X}(d,r,\alpha)\,\longrightarrow \,\cU
\,=\, \bigoplus_{i=1}^{r}H^{0}(X,K^{i}((i-1)D))
$$
be the restriction to the cotangent bundle of the moduli space of stable bundles
of the Hitchin integrable system in (\ref{eqn:hitchin}).
To each point $u\,\in\, \cU$ we associated its spectral curve $X_u\,\subset\, \cS$.

\begin{lem}\label{lem:U}
If $g \geq 2$, then the Zariski open subset $U$ of $\cU$ that parametrizes
the smooth spectral curves is non-empty.
\end{lem}

\begin{proof}
If $K^{r}((r-1)D)$ has a section without multiple zeros, then the above open subset $U$ is 
nonempty (cf. \cite[Remark 3.5]{BNR}). A holomorphic
line bundle on $X$ of degree at least $2g+1$ is 
very ample (cf. \cite[IV Corollary 3.2]{Ha}), and hence $U$ is non-empty whenever
$r(2g-2)+(r-1)n\,\ge\, 2g+1$, and this holds when $g\geq 2$.
\end{proof}

Define $\cJ\,:=\,H^{-1}(U)$, where $H$ is the Hitchin map for the
moduli of Higgs bundles \eqref{eqn:hitchin} and  
$U$ is the open subset in Lemma \ref{lem:U}. Let 
$$
H_{\cJ}\,:\,\cJ\,\longrightarrow\, U
$$
be the restriction of $H$. Let
$$
H_{\cS}\,:\,\cS\, \longrightarrow\, U
$$
be the total space for the family of spectral curve over $U$, so that the fiber of
$H_{\cS}$ over any $u\,\in\, U$ is the spectral curve $X_{u}$.

As we have seen in Corollary \ref{MainCor}, the surface $Q$ given by the Theorem of Hurtubise in this setting is the image of $\mathcal{K}$ in
$\mathcal{K(D)}$ under the injective morphism of sheaves
$K\longrightarrow K(D)$. In particular, $Q$ is singular.

The moduli space of parabolic Higgs bundles is known to be a
K\"{a}hler manifold 
provided with a $\CC^{\ast}$ action whose restriction to
an ${\rm S}^1$ action  preserves the K\"ahler structure 
\begin{eqnarray}\label{eqn:c-action}
\tau:\CC^{\ast} \times \cM_{X}(r,d,\alpha) &\longrightarrow& \cM_X(r,d,\alpha)\\
(t\, ,(E\, ,\Phi))&\longmapsto& (E\, ,t\Phi).\notag
\end{eqnarray}
This $\mathbb{C}^*$ action is compatible with scalar multiplication in the fibers of
the cotangent bundle $T^\ast M_X(r,d,\alpha)$ under the inclusion of
this cotangent bundle in the moduli of parabolic Higgs bundles.
It induces a $\CC^{\ast}$ action on $\cS$:
\begin{eqnarray}\label{eq:CQ}
\CC^{\ast} \times \cS &\longrightarrow& \cS\\
(t\, ,x\in X_{u})&\longmapsto& (tx \in X_{t\cdot u}),\notag
\end{eqnarray}
where $t\cdot (s_1\, , \cdots\, , s_r)\,=\, 
(t s_1\, , t^2 s_2\, ,  \cdots\, , t^{r}s_r)$
(see \eqref{eq:char_pol}), and the multiplication $tx$ is defined
using the embedding of the spectral curve $X_u$ in the total space of
$K(D)$. 
This action of $\CC^{\ast}$ on $\cS$ evidently produces an
action of $\CC^{\ast}$ on the quotient surface $Q$. Let 
$$
Q^{\CC^{\ast}}\subset Q
$$
be the fixed point locus for the above $\CC^{\ast}$ action on $Q$.

\begin{lem}
The subset $Q^{\CC^{\ast}}$  is the zero section of the fibration
$\mathcal{K(D)}\longrightarrow X$.
\end{lem}

\begin{proof}
Since the natural inclusion $K \, \hookrightarrow\, K(D)$ of
${\mathcal O}_X$--modules commutes with the multiplicative action of
$\CC^{\ast}$, the surface $Q$, which is the image of the total space
of $K$ in $\mathcal{K}(\mathcal{D})$, is preserved by the action of
$\CC^{\ast}$ on $\mathcal{K}(\mathcal{D})$. Therefore, the action of
$\CC^{\ast}$ on $\mathcal{K}(\mathcal{D})$ produces an action of
$\CC^{\ast}$ on $Q$. 
This action of $\CC^{\ast}$ on $Q$ coincides
with the action on $Q$ induced by (\ref{eq:CQ}). The lemma follows
from this.
\end{proof}

\begin{cor}\label{lem:curve}
The curve $X$ coincides with $Q^{\CC^{\ast}}$.
\end{cor}

\begin{prop}\label{prop:points}
The set of parabolic points coincides with the subset of $Q^{\CC^{\ast}}$ through which every spectral cover pass.
\end{prop}

\begin{proof}
Since the residue of $\Phi$ on the parabolic points is nilpotent, 
all spectral curves $X_{u}$ totally ramify
over the parabolic points, and they intersect the fibre over the
parabolic points at zero. 

Conversely, let $x\in X$ be a point
which is not parabolic. There exists a section $s_r\in H^0(K^r((r-1)D))$
which does not vanish at $x$ since this linear systems is
base point free (recall that we are assuming $g\geq 2$). 
Furthermore, this section has still no zero on $x$ when considered
as a section of $H^0(K^r(rD))$ because $x$ is not a parabolic point.
Therefore, the spectral curve 
$\widetilde{x}^r +s_r=0$ on $\mathcal{K(D)}$ 
intersects the fibre over $x$ away from
zero, and the spectral curve $\widetilde{x}^r=0$ intersects
it only at zero, so there is no point over the fibre of $x$ through
which every spectral cover passes.
\end{proof}

\subsection{Proof of Theorem \ref{theoremPH}}

We are given the moduli space 
$\cM$ as an abstract algebraic variety with a 
holomorphic symplectic form, a line bundle 
$\cL$ and a algebraic $\mathbb{C}^*$ action on $\cM$ with a linearization on 
$\cL$. Looking at global functions on 
$\cM$ 
$$
\alpha:\cM \longrightarrow \Spec\Gamma(\cM)
$$
we obtain a morphism $\alpha$
which is isomorphic to the Hitchin fibration 
(cf. \eqref{eq:globalsections}) and the fibers are 
Lagrangians with respect to the given holomorphic symplectic
form. The subset $U\subset
\Spec\Gamma(\cM)$ of points
corresponding to smooth spectral curves can be recovered as the points
whose fibers are abelian varieties. Let $\beta$ be the restriction
of $\alpha$ over $U$
$$
\xymatrix{
{\cJ} \ar@{^{(}->}[r] \ar[d]_{\beta}& {\cM} \ar[d]^{\alpha} \\
{U} \ar@{^{(}->}[r] & {\Spec\Gamma{\cM}}\\
}
$$
The line bundle $\cL$ restricts
to a (multiple of) a principal polarization on these abelian
varieties, and then the classical Torelli theorem gives us a family
of curves $\cS\longrightarrow U$, such that the fiber $\cJ_u$ over $u\in U$ is
the Jacobian of $\cS_u$. 
Locally on $U$ there is an Abel-Jacobi map
$I:\cS\longrightarrow \cJ$.

The $\mathbb{C}^*$ action on $\cM$ restricts to a $\mathbb{C}^*$
action on the family of Jacobians $\cJ$. This family of Jacobians has a 
family of principal polarizations given by the line bundle $\cL$. 
The $\mathbb{C}^*$ action has a lift to $\cL$ hence we have an action
on the family of principal polarized Jacobians.

By the proof given by Weil of the 
Torelli theorem \cite[Hauptsatz, p.~35]{We}, an
isomorphism $\psi:(\cJ_u,\theta_u)\longrightarrow (\cJ_{u'},\theta_{u'})$ 
of principal polarized Jacobians induces an isomorphism 
$f:\cS_u\longrightarrow \cS_u'$ of the corresponding curves, and this provides
an action of $\mathbb{C}^*$ on the family of curves $\cS$.

Now we apply Corollary \ref{MainCor}
to obtain a surface $Q$ as a quotient of $\cS$. The action on 
$\cS$ we have just defined clearly coincides with the action given
in \eqref{eq:CQ}, therefore by Corollary \ref{lem:curve} we recover
$X$, and by Proposition \ref{prop:points} we recover the parabolic
points $D$, thus proving our main theorem.

\subsection{Proof of Theorem \ref{theoremP}}

We are given the moduli space as a smooth algebraic variety $M$ with a line
bundle $\cL$. We consider the total space of the cotangent bundle 
$T^*M$. This has a canonical holomorphic symplectic structure, and
a $\mathbb{C}^*$ given by scalar multiplication on the fibers. 
The pullback of the line bundle to $T^* M$ is trivial along the
fibers, so there is a canonical lift of the $\mathbb{C}^*$ action
to the pullback of the line bundle $\cL$ to $T^* M$.

We claim that the generic fiber of the morphism given by global
sections
$$
h : T^* M \longrightarrow \Spec(\Gamma(T^* M))
$$ 
is an open subset of an abelian variety.
Indeed, we know that $M$ is the moduli space for some algebraic curve $X$
(which we want to find), so we know that $T^* M$ is an open subset
of a moduli space of parabolic Higgs bundles $\cM$, and by 
Corollary \ref{codimfibre} we know that the codimension of the 
complement of this open set is at least two. Therefore, global
section on $T^*M$ extend uniquely to global sections on $\cM$ and
the morphism $h$
is the restriction of the morphism of global sections
of some moduli space of Higgs bundles $\cM$
$$
\xymatrix{
{T^* M} \ar[r]^-{h} \ar@{^{(}->}[d] & {\Spec \Gamma (T^* M)} \ar@{=}[d]\\
{\cM} \ar[r]^-{H}  & {\Spec \Gamma (\cM)} \\
}
$$
The compactification of the fiber over $u$ to an abelian variety
is unique, because birational abelian varieties are isomorphic.
Therefore, the isomorphism class of $\cJ:=H^{-1}(U)$ is uniquely
defined by the isomorphism class of $M$, and does not depend on 
the choice of $\cM$.

Since the codimension of the complement of the inclusion 
$T^* M|_U\subset \cJ$ is at least two, all the structure that
we have on $T^* M$ extends uniquely to $\cJ$, namely the 
determinant line bundle $\cL$, the $\mathbb{C}^*$ action with
the lift to $\cL$ and the holomorphic symplectic form.
Therefore we can now use the same arguments as in the proof
of the main theorem to recover the curve $X$ and the parabolic
points.

\section{Codimension computation}

In this section we compute the codimension of the complement of 
$T^{\ast}M(d,r,\alpha) $ inside $\cM(d,r,\alpha)$ 
fiber-wise following the arguments in \cite[Section 5]{biswas-gothen-logares}.
This complement is 
$$
\cV=\{(E,\Phi)\in \cM(d,r,\alpha)\,\mid\, E\,\mathrm{is\, not\,
  stable}\,\}  \; .
$$

Recall from (\ref{eqn:c-action}) that the moduli space of parabolic
Higgs bundles is known to be a K\"{a}hler manifold provided with a
$\CC^{\ast}$ action, whose restriction to a ${\rm S}^1$ action 
preserves the K\"{a}hler structure. 

This action provide us with two stratifications of the moduli space. The first one is the
Bia{\l}ynicki-Birula stratification consisting of subsets of $\cM(d,r,\alpha)$ such 
$$
U^{+}_{\lambda}:=\{p\in \cM_X(d,r,\alpha) ; \lim_{t\rightarrow 0} tp\in F_{\lambda}\}
$$
and 
$$
U^{-}_{\lambda}:=\{p\in \cM_X(d,r,\alpha) ;\lim_{t\rightarrow \infty}tp\in F_{\lambda}\}
$$
where $F_\lambda$ are the disjoint connected components of the fixed pointed set $F$ for the $\CC^{\ast}$-action on $\cM(d,r,\alpha)$.

The second one is known as the Morse stratification and comes from the restriction of the $\CC^{\ast}$-action to an $S^{1}$-action. The last also preserves the K\"{a}hler form, hence it give us a circle Hamiltonian action on $\cM(d,r,\alpha)$ with associated moment map,
\begin{eqnarray*}
\mu:\cM_X(r,d,\alpha)&\longrightarrow& \RR\\
(E,\Phi)&\longmapsto & ||\Phi||^{2}
\end{eqnarray*}
which is proper, bounded below and has a finite number of critical submanifolds. So this map is a Morse-Bott map.

For any component $F_{\lambda}$, we recall the definition of the upwards Morse strata, $\widetilde{U}^{+}_{\lambda}$, and the downwards Morse strata, $\widetilde{U}^{-}_{\lambda}$, that is
$$
\widetilde{U}^{+}_{\lambda}:=\{p\in \cM_X(d,r,\alpha); \lim_{t\rightarrow -\infty} \psi_{t}(p)\in F_{\lambda}   \}
$$
and
$$
\widetilde{U}^{-}_{\lambda}:=\{ p\in \cM_X(d,r,\alpha); \lim_{t\rightarrow +\infty} \psi_{t}(p)\in F_{\lambda}   \}
$$
Recall that this stratifications were proven to be equal $U^{+}=\widetilde{U}^{+}$ and $U^{-}=\widetilde{U}^{-}$ by Kirwan in \cite[Theorem 6.16]{kirwan}.

The union $N=\bigcup_{\lambda}\widetilde{U}^{-}_{\lambda}$ is known as \emph{downwards Morse flow}. 

The inverse over the $0$ point of the Hitchin map $H^{-1}(0)$ is called \emph{nilpotent cone}, and it coincides with the downwards Morse flow, i.e. $N=H^{-1}(0)$ \cite[Theorem 3.13]{garcia-prada-gothen-munoz}. 

The following proposition takes the same steps as Proposition 5.1 in \cite{biswas-gothen-logares} provided that in a family of parabolic bundles the Harder-Narasimhan type increases under specialization.

\begin{prop}\label{V=V'}
Let $\cV$ be the complement of the cotangent bundle of
$M_{X}(d,r,\alpha)$ in $\cM_{X}(d,r,\alpha)$ and let $\cV'$ be the
Bia{\l}ynicki-Birula flow which does not converge 
to $M_X(d,r,\alpha)$
,  that is
$$\cV= \{(E,\Phi)\in\cM_{X}(d,r,\alpha) : E \,\mathrm{is\, not\,
  stable}\}\quad \mathrm{and}\quad \cV'=
\{(E,\Phi): \lim_{t\rightarrow 0} (E,t\Phi)\notin M_X(d,r,\alpha)\}.$$
  Then 
  $$ \cV'=\cV.$$
\end{prop}

\begin{proof}
Let $(E,\Phi)\notin \cV$ that is $E$ is stable, then
$\lim_{t\rightarrow 0}(E,t\Phi)=(E,0)\in M_X(d,r,\alpha)$, so it proves that
$\cV'\subset \cV$.  To prove the converse, take $(E,\Phi)$ where $E$
is not stable. There exists a Harder-Narasimhan filtration for $E$,
that is
$$
E=E_{m}\supset E_{m-1}\supset \cdots\supset E_{1}\supset 0.
$$
Following Atiyah and Bott \cite{AB} we define the type of the
Harder-Narasimhan filtration, as the following vector
$(\mu_{1},\ldots,\mu_{r})$ where $\mu_{i}=\deg(F_{i})/\rk(F_{i})$ and
$F_{i}=E_{i}/E_{i+1}$.

In a family of parabolic bundles Nitsure \cite[Proposition
1.10]{nitsure} proved that the Harder-Narasimhan type increases under
specialization. Hence, as the Hitchin map is proper, and its
composition with the map given by the $\CC^{\ast}$-action, is such
that $\lim_{t\rightarrow 0}h(E,t\Phi)=0$ then $\tau$ extends to a
morphism
$$
\widetilde{\tau}:\CC\times \cM(d,r,\alpha) \longrightarrow \cM(d,r,\alpha).
$$
This maps gives, by pullback, that a family over $\CC^{\ast}$ of non
semistable parabolic bundles specializes to a non semistable parabolic
bundle. That is, for $\EE$, the universal bundle over
$\cM(d,r,\alpha)\times X$ when pullback $(\widetilde{\tau}\times
1_{X})^{\ast}(\EE)$ gives a family of bundles over $X$ parametrized by
$\CC$. If $(\widetilde{\tau}\times 1_{X})^{\ast}(\EE)|_{t}\times X=E$
for $t\neq 0$, is not semistable, then the specialization result says
that $(\widetilde{\tau}\times 1_{X})^{\ast}(\EE)|_{{0}\times X}$ is
not semistable. This completes the proof.
\end{proof}

The following facts are recovered from the literature on
parabolic Higgs bundles.

Let $(E,\Phi)$ be a fixed point for the circle action, we have an
isomorphism $(E,\Phi)\cong (E,e^{\imat\theta}\Phi)$ for $\theta\in
[0,2\pi)$ yielding the following commutative diagram.
$$
\xymatrix{
E \ar[r]^{\hspace*{-20pt}\Phi}\ar[d]^{\psi_{\theta}} & E\otimes K(D)\ar[d]^{\psi_{\theta}\otimes 1_{K(D)}}\\
E \ar[r]^{\hspace*{-20pt}e^{\imat\theta}\Phi} & E\otimes K(D).
}
$$

\begin{prop}[{\cite[Theorem 8]{simpson}}]\label{fixed_points:decomposition}
If $(E,\Phi)$ belongs to a critical subvariety $F_{\lambda}$ for the circle action on $\cM_{X}(d,r,\alpha)$ then $E$ splits
$$
E=\bigoplus_{l=0}^{m} E_{l}
$$ 
and $\Phi\in H^{0}(\SPH(E_{l},E_{l+1})\otimes K(D))$. 
\end{prop}
The parabolic Higgs bundle in this case $(E,\Phi)$ is called Hodge bundle.

The deformation theory  of the moduli space of parabolic Higgs bundles was worked out in \cite{y}. It is given by the following complex of bundles, 
$$
C^{\bullet}(E): \PE(E)\stackrel{\Phi:=[\cdot,\Phi]}{\longrightarrow} \SPE(E)\otimes K(D).
$$

The tangent space of the moduli space $\cM_X (d, r, \alpha)$ at a stable point $(E,\Phi)$ is
then the first cohomology group $\HH^{1} (C^{\bullet}(E))$ of this complex. Hence for a fixed point $(E,\Phi)$ of the $\CC^{\ast}$-action,  the decomposition in Proposition \ref{fixed_points:decomposition} induces a decomposition of the deformation complex and of the tangent space at the fixed point. That is, we define
$$
C_{k}:= \bigoplus_{j-i=k}\PH(E_{i},E_{j}) \quad \mathrm{and} \quad \widehat{C}_{k+1}:=\bigoplus_{j-i=k}\SPH(E_{i},E_{j})
$$
so then
$$
C^{\bullet}(E)_{k}: C_{k}\stackrel{\Phi_{k}}{\longrightarrow} \widehat{C}_{k+1}\otimes K(D), 
$$
and
$$
C^{\bullet}(E)=\bigoplus_{k=-m-1}^{k=m} C^{\bullet}(E)_{k}.
$$

For this deformation complex, there is a long exact sequence 
\begin{eqnarray*}
&0&\rightarrow \HH^{0}(C^{\bullet}(E)_k)\rightarrow H^{0}(\bigoplus_{j-i=k}\PH(E_{i},E_{j})\rightarrow H^{0}(\bigoplus_{j-i=k}\SPH(E_{i},E_{j})\otimes K(D))\\
&&\rightarrow \HH^{1}(C^{\bullet}(E)_k)\rightarrow H^{1}(\bigoplus_{j-i=k}\PH(E_{i},E_{j})\rightarrow H^{1}(\bigoplus_{j-i=k}\SPH(E_{i},E_{j})\otimes K(D))\\
&&\rightarrow \HH^{2}(C^{\bullet}(E)_{k})\rightarrow 0.
\end{eqnarray*}

\begin{prop}[{\cite[Proposition 3.9]{garcia-prada-gothen-munoz}}]\mbox{}
\begin{itemize}
\item There is a natural isomorphism 
$$
\HH^{1}(C^{\bullet}(E)_{k})\simeq \HH^{1}(C^{\bullet}(E)_{-k-1})^{\ast}
$$
and hence a natural isomorphism
$$
T_{(E,\Phi)}\cM_{k}\simeq (T_{(E,\Phi)}\cM_{1-k})^{\ast}
$$
\item If $(E,\Phi)$ is stable, then we have 
$$
\HH^{0}(C^{\bullet}(E)_{k})=\left\lbrace \begin{array}{ll} \CC & if \; k=0\\ 0 & otherwise,\end{array}\right.
$$
and
$$
\HH^{2}(C^{\bullet}(E)_{k})=\left\lbrace \begin{array}{ll} \CC & if \; k=-1\\ 0 & otherwise.\end{array}\right.
$$
\end{itemize}
\end{prop}
\hfill$\Box$

Hence, 
\begin{lem}
$$
\dim \HH^{1}(C^{\bullet}(E)_{k})=\left\lbrace \begin{array}{ll} 1-\chi(C^{\bullet}(E)_{k}) & if \; k=0\\ -\chi(C^{\bullet}(E)_{k}) & otherwise.\end{array}\right.
$$
\end{lem}
\hfill$\Box$

\begin{thm}[{\cite[Theorem 3.8]{garcia-prada-gothen-munoz}}]
The function $\mu:\cM_{X}(r,d,\alpha)\longrightarrow \RR$ defined by $\mu(E,\Phi)=\| \Phi\|^{2}$ is a perfect Bott--Morse function. A parabolic Higgs bundle represents a critical point of $\mu$ if and only if it is a parabolic complex variation of Hodge structure, i.e. $E=\bigoplus _{k=0}^{m} E_k$ with $\Phi_{k}=\Phi|_{E_{k}}:\,E_{l}\longrightarrow E_{k+1}\otimes K(D)$ strongly parabolic (where $\Phi=0$ if and only if $m=0$). The tangent space to $\cM_{X}(r,d,\alpha)$ at a critical point $(E,\Phi)$ decomposes as 
$$
T_{(E,\Phi)}\cM_{X}(r,d,\alpha)=\bigoplus_{k=-m}^{m+1} T_{(E,\Phi)}\cM_{X}(r,d,\alpha)_{k}
$$ 
where the eigenvalue $k$ subspace of the Hessian of $\mu$ is
$$
T_{(E,\Phi)}\cM_{X}(r,d,\alpha)_{k} \cong \HH^{1}(C^{\bullet}(E)_{k} ).
$$

\end{thm}
\hfill $\Box$

For a critical point $(E,\Phi)$ of $\mu$ we denote by $T_{(E,\Phi)}\cM_{X}(r,d,\alpha)_{<0}$ the subspace of the tangent space on which the Hessian of $\mu$ has negative eigenvalues. The real dimension of this subspace is called the Morse index at the point $(E,\Phi)$.

\begin{prop}\label{prop:codim}
The codimension of the complement of 
$T^\ast M_{X}(d,r,\alpha)$ in $\cM_X(d,r,\alpha)$ is equal to half of the minima of the Morse indexes at points $(E,\Phi)\in F_{\lambda}$ for $\lambda\neq 0$.
\end{prop}

\begin{proof}
The complement of 
$T^\ast M_X(d,r,\alpha)$ is equal to $\cV$ which is also equal to
$\cV'$ from Proposition \ref{V=V'}. Bott--Morse theory give us that
$\cV'=\bigcup_{\lambda\neq 0}U^{+}_{\lambda}$, so we conclude that
$$
\codim(\cV)=\min_{\lambda \neq 0} \codim U^{+}_{\lambda}.
$$
{}From Bott--Morse theory we also know that the dimension of the upwards Morse flow is such that
$$
\dim U^{+}_{\lambda}+\dim T_{(E,\Phi)}\,\cM_{X}(d,r,\alpha)_{<0}=\dim \cM_{X}(d,r,\alpha),
$$
where $T_{E}\cM_{X}(d,r,\alpha)_{<0}$ is the negative eigenspace for the Hessian of the perfect Bott--Morse function $\mu$ for an $E\in U^{+}_{\lambda}$. As the Morse index $\mu_{\lambda}= 2 \dim T_{E}\cM_{X}(d,r,\alpha)_{<0}$,
$$
\codim U^{+}_{\lambda}= \frac{1}{2}\mu_{\lambda}.
$$
Our statement is then
$$
\codim (\cV)=\min_{\lambda\neq 0} \frac{1}{2}\mu_{\lambda}
$$
that is, 
$$
\codim (\cV)=\min_{\lambda\neq 0} \dim T_{E}\cM_{X}(d,r,\alpha)_{<0}
$$

\end{proof}

\begin{lem}
$$
T_{(E,\Phi)}\cM_{X}(d,r,\alpha)_{<0}=\sum_{k>0} -\chi(C^{\bullet}(E)_{k}.
$$
\end{lem}
\hfill$ \Box$

So, we need to bound the Euler characteristic for any $k$. 

\begin{prop}\label{prop:chi}
$$
-\chi(C^{\bullet}(E)_{k})\ge (g-1)(\rk(C_{k}-\rk(\widehat{C}_{k+1}))
$$
\end{prop}
\begin{proof}
Recall that 
\begin{small}
\begin{eqnarray}\label{eq:euler}
\chi(C^{\bullet}(E)_{k})&=&\dim H^{0}(C_{k})-\dim H^{1}(C_{k})-\dim H^{0}(\widehat{C}_{k+1}\otimes K(D))+\dim H^{1}(\widehat{C}_{k+1}\otimes K(D))\notag\\
&=&\deg(C_{k})-\deg(\widehat{C}_{k+1})-\rk(\widehat{C}_{k+1})\deg(K(D))+(\rk(C_{k})-\rk(\widehat{C}_{k+1}))(1-g).
\end{eqnarray}
\end{small}

We first bound $\deg(C_k)-\deg(\widehat{C}_{k+1})$. Consider the following short exact sequences of bundles,
\begin{eqnarray*}
&&0\longrightarrow \ker(\Phi_{k})\longrightarrow  C_k \longrightarrow \img(\Phi_{k})\longrightarrow 0\\
&&0\longrightarrow \img(\Phi_{k})\longrightarrow \widehat{C}_{k+1}\otimes K(D) \longrightarrow \coker(\Phi_{k})\longrightarrow 0.
\end{eqnarray*}
then
\begin{equation}\label{eq:bound1}
\deg(C_{k})-\deg(\widehat{C}_{k+1})=\deg(\ker(\Phi_{k}))+\deg(K(D))\rk(\widehat{C}_{k+1})-\deg(\coker(\Phi_{k})).
\end{equation}

The $\ker(\Phi_{k})\subset\PE(E)$ is a subbundle of the bundle of parabolic endomorphisms of $E$, which we claim  is semistable whenever $E$ is stable (see Lemma \ref{lem:ss}). Hence $\pdeg(\ker(\Phi_{k}))\le 0$ and this implies $\deg(\ker(\Phi_{k}))\le 0$.

Hence
$$
\deg(C_{k})-\deg(\widehat{C}_{k+1})\le \deg(K(D))\rk(\widehat{C}_{k+1})-\deg(\coker(\Phi_{k})).
$$
We also get that  
\begin{equation}\label{eq:coker}
-\deg(\coker(\Phi_{k}))\le (2-2g)(\rk(\widehat{C}_{k+1})-\rk(\Phi_{k})),
\end{equation}
so that 

\begin{equation}\label{eq:bound-deg}
\deg(C_{k})-\deg(\widehat{C}_{k+1})\le n\rk(\widehat{C}_{k+1})+(2g-2)\rk(\Phi_{k})
\end{equation}
in the following way.

Note that for any two parabolic bundles $E$, $F$, then $\PH(E,F)^{\ast}=\SPH(F,E)\otimes \cO(D)$. So then $C_{k}^{\ast}=(\bigoplus_{j-i=k} \PH(E_{i},E_{j}))^{\ast}=\bigoplus_{j-i=k} \SPH(E_{j},E_{i})\otimes \cO(D)=\widehat{C}_{k}\otimes\cO(D)$.

Consider the adjoint map
\begin{equation}
\Phi_{k}^{t}:(\widehat{C}_{k+1}\otimes K(D))^{\ast}\longrightarrow (C_{k})^{\ast}
\end{equation}
then
$$
\ker(\Phi_{k}^{t})\hookrightarrow (\widehat{C}_{k+1}\otimes K(D))^{\ast}\cong C_{-1-k}\otimes K^{-1}.
$$ 
Dualizing again we get a surjective homomorphism,
$$
\widehat{C}_{k+1}\otimes K(D)\longrightarrow (\ker(\Phi_{k}^{t}))^{\ast}
$$

Define the homomorphism 
$$
f: \coker(\Phi_{k})\longrightarrow \ker(\Phi_{k}^{t})^{\ast}
$$
which makes the following diagram commutative
$$
\xymatrix{
0\ar[r]&\img(\Phi_{k})\ar[r] & \widehat{C}_{k+1}\otimes K(D)\ar[r]\ar@{=}[d]&\coker(\Phi_{k})\ar[r] \ar[d]^{f}&0\\
0\ar[r]&(\img(\Phi_{k}^{t}))^{\ast}\ar[r]&  \widehat{C}_{k+1}\otimes K(D)\ar[r]&(\ker(\Phi_{k}^{t}))^{\ast}\ar[r]&0
}
$$
Note that $f$ is surjective and $\ker(f)$ is a torsion subsheaf. Hence,
$$
0\longrightarrow \ker(f)\longrightarrow\coker(\Phi_{k})\longrightarrow (\ker(\Phi_{k}^{t}))^{\ast}\longrightarrow 0,
$$
and 
$$
\deg(\coker(\Phi_{k}))\ge \deg(\ker(\Phi_{k}^{t})^{\ast}).
$$
As $\ker(\Phi_{k}^{t})$ is a sub bundle of $C_{k+1}\otimes K$, 
\begin{equation}\label{eq:bound2}
-\deg(\coker(\Phi_{k}))\le \deg(\ker(\Phi_{k}^{t}))
\end{equation}

Note that there are isomorphisms, making the following diagram commutative
$$
\xymatrix{
(\widehat{C}_{k+1}\otimes K(D))^{\ast} \ar[d]^{\cong} \ar[r]^{\Phi^{t}_{k}} &
(C_{k})^{\ast}\ar[d]^{\cong}\\
C_{-1-k}\otimes K^{-1}\ar[r]_{\Phi_{-1-k}\otimes 1_{K^{-1}}} &
\widehat{C}_{-k}\otimes \cO(D),
}
$$
therefore
$$
\Phi_{k}^{t}\cong \Phi_{-1-k}\otimes 1_{K^{-1}}
$$
so 
$$
\ker(\Phi^{t}_{k})=\ker(\Phi_{-1-k})\otimes K^{-1}
$$
and
$$
\deg(\ker(\Phi^{t}_{k}))=\deg(\ker(\Phi_{-1-k}))+(2-2g)\rk(\ker(\Phi_{-1-k}))
$$
Notice that $\rk(\Phi_{-1-k})=\rk(\Phi^{t}_{k})=\rk(\Phi_{k})$ and
$\rk(\widehat{C}_{k+1})=\rk((\widehat{C}_{k+1})^{\ast})=\rk(C_{-1-k})$. Then
$\rk(\ker(\Phi_{-1-k}))=\rk(\widehat{C}_{k+1})-\rk(\Phi_{k})$, so that, equation
(\ref{eq:bound2}) becomes
$$
-\deg(\coker\Phi_{k})\le \deg(\ker (\Phi_{-1-k}))+(2-2g)(\rk(\widehat{C}_{k+1})-\rk(\Phi_{k}))
$$
and finally, by stability (see Lemma \ref{lem:ss}),
$$
-\deg(\coker\Phi_{k})\le (2-2g)(\rk(\widehat{C}_{k+1})-\rk(\Phi_{k})).
$$
This provides equation (\ref{eq:coker}).

Putting together equations (\ref{eq:coker}) and (\ref{eq:bound-deg}) we get
$$
\chi(C^{\bullet}(E)_{k})\le (1-g)(\rk(C_{k})-\rk(\widehat{C}_{k+1})),
$$
hence, 
$$
-\chi(C^{\bullet}(E)_{k})\ge (g-1)(\rk(C_{k})-\rk(\widehat{C}_{k+1})), 
$$
as we wanted.
\end{proof}

\begin{lem}\label{lem:ss}
Let $(E,\Phi)$ be a stable parabolic Higgs bundle, then $(\PE(E),\ad(\Phi))$ is semistable.
\end{lem}

\begin{proof}
The proof follows the arguments in \cite[Proposition 6.7]{garcia-prada-logares-munoz} adapted to the parabolic situation. That is, the vector bundle $\PE(E)$ has a natural parabolic structure
induced by the parabolic structure of $E$. In fact $\PE(E)$ as a
parabolic bundle is the parabolic tensor product of the parabolic
bundle $E$ and the parabolic dual of $E$ (see \cite{y}), and
hence its parabolic degree is $0$. With respect to this parabolic
structure $(\PE(E),\ad(\Phi))$, where $\ad(\Phi): \PE(E)\to
\SPE(E)\otimes K(D)$, is, again, a parabolic Higgs bundle. Now, the
stability of $(E,\Phi)$ implies the polystability of
$(\PE(E),\ad(\Phi))$. 
\end{proof}

\begin{prop}
$$
\codim (\cV)\ge \frac{1}{2}(r-1)(g-1).
$$
\end{prop}

\begin{proof}
Propositions \ref{prop:codim} and \ref{prop:chi} give 
$$
\codim (\cV)\,=\,\min_{F_{\lambda}}\left\{ \sum_{k>0} -\chi(C^{\bullet}(E)_{k})\right\}
$$
$$
\ge \,\min_{F_{\lambda}}\left\{\sum_{k>0}(g-1)(\rk(C_{k}-\rk(\widehat{C}_{k+1}))\right\}
\,=\, \rk(C_{1})(g-1)\, ,
$$
where $\rk(C_{1})=\rk(\oplus_{j-i=1}\PH(E_{i},E_{j})$ so if we denote $r_{i}=\rk(E_{i}$ the rank of each piece is $\rk(\PH(E_{i},E_{j})=\frac{1}{2}r_{i}r_{j}$, so then $\rk(C_{1})=\frac{1}{2} (r_{1}r_{2}+\cdots +r_{m-1}r_{m})$ which is definitely $\rk(C_{1})\ge \frac{1}{2}(r-1)$.

\end{proof}

\begin{cor}\label{codimfibre}
For $g= 2$ and $r\ge 5$ or $g= 3$ and $r\ge 3$
$g\ge 5$ and $r\ge 2$, 
and $u$ a generic point in $\cU$, the codimension of the fiber
$h^{-1}_0(u)$ in $h^{-1}(u)$ is greater or equal to 2.
\end{cor}

\begin{rmk}
The codimension does not depend on the number of marked points, as in \cite{biswas-gothen-logares} it did not depend on the degree of the line bundle $L$, which was twisting the Higgs bundle and in this case is $K(D)$.
\end{rmk}

We also obtain the following.

\begin{cor}
For $g= 2$ and $r\ge 5$ or $g= 3$ and $r\ge 3$
$g\ge 5$ and $r\ge 2$, 
the moduli space of
parabolic bundles has the same number of irreducible components as the
moduli space of parabolic Higgs bundles.
\end{cor}


\begin{thebibliography}{ZZZZ}

\bibitem[AB]{AB} M.F. Atiyah and R. Bott, The Yang-Mills equations over
Riemann surfaces, \emph{Phil. Trans. R. Soc. London} \textbf{308}
(1982) 523--615.

\bibitem[BBB]{BBB} V. Balaji, S. del Ba\~{n}o and I. Biswas, A Torelli type theorem
for the moduli space of parabolic vector bundles over curves,
\emph{Math. Proc. Cambridge Philos. Soc.} \textbf{130} (2001), 269--280.

\bibitem[BG]{BG} I. Biswas and T.L. G\'omez, 
A Torelli theorem for the moduli space of Higgs bundles on a curve, 
\emph{Quart. Jour. Math.} \textbf{54} (2003), 159--169.

\bibitem[BGL]{biswas-gothen-logares} 
I. Biswas, P.B. Gothen and M. Logares, On moduli spaces of Hitchin pairs, \emph{Math.
Proc. Cambridge Philos. Soc.} \textbf{151} (2011), 441–-457.

\bibitem[BHK]{BHK} I. Biswas, Y. Holla and C. Kumar, On moduli spaces of parabolic
vector bundles of rank 2 over $\CC\PP^{1}$, \emph{Michigan Math. Jour.} \textbf{59}
(2010), 467--479. 

\bibitem[BNR]{BNR} A. Beauville, M.S. Narasimhan and S. Ramanan, 
Spectral curves and the generalized theta divisor,
\emph{J. reine angew. Math.} \textbf{398}, (1989), 169--179.

\bibitem[CRS]{crs} C. Ciliberto, P. Ribenboim and E. Sernesi,
Collected papers of  
Ruggiero Torelli, \emph{Queen's Papers in Pure and Applied 
Mathematics}, 101. Queen's  
University, Kingston, 1995.

\bibitem[GGM]{garcia-prada-gothen-munoz} 
O. Garc\'{\i}a-Prada, P.B. Gothen and V. Mu\~{n}oz, Betti numbers for the moduli space
of rank $3$ parabolic Higgs bundles, \emph{ Mem. Amer. Math. Soc.} \textbf{187} (2007),
no. 879.

\bibitem[GLM]{garcia-prada-logares-munoz}
O. Garc\'{\i}a-Prada, M. Logares and V. Mu\~{n}oz, Moduli spaces of parabolic $\U(p,
q)$-Higgs bundles, \emph{Quart. Jour. Math.} \textbf{60} (2009), 183--233.

\bibitem[GL]{GL} T.L. G\'omez and M. Logares,
Torelli theorem for the moduli space of parabolic Higgs bundles,
\emph{Adv. Geom.} \textbf{11} (2011), 429–-444.

\bibitem[KM]{KM} F. Knudsen and D. Mumford, 
The projectivity of the moduli space of stable curves I:
preliminaries on ``det'' and ``div''.
\emph{Math. Scand.}, \textbf{39} (1976), 19--55.

\bibitem[Har]{Ha}R. Hartshorne, 
\textit{Algebraic Geometry}, Graduate Texts in Mathematics, No. 52.
Springer-Verlag, New York-Heidelberg, 1977.

\bibitem[Hau]{hausel} T. Hausel, 
Compactification of the moduli of Higgs bundles, \emph{Jour. Reine Angew. Math.}
\textbf{503} (1998), 169--192.

\bibitem[Hi]{Hi} N.J. Hitchin, Stable bundles and integrable systems, 
\emph{Duke Math. Jour.} \textbf{54} (1987), 91--114.

\bibitem[Hu]{Hu} J.C. Hurtubise, Integrable systems and algebraic surfaces, 
\emph{Duke Math. Jour.} \textbf{83} (1996), 19--49.

\bibitem[Ki]{kirwan} F.C. Kirwan, \textit{Cohomology of quotients in symplectic and
algebraic geometry}, Mathematical Notes 31, Princeton University Press, 1984.

\bibitem[LM]{LM} M. Logares and J. Martens, Moduli of parabolic Higgs bundles and
Atiyah algebroids, \emph{Jour. reine angew. Math.} \textbf{649} (2010), 89--116.

\bibitem[MN]{MN} D. Mumford and P. Newstead, Periods of a moduli space of bundles on
curves, \emph{Amer. Jour. Math.} \textbf{90} (1968), 1200--1208.

\bibitem[Ni]{nitsure} 
N. Nitsure, Cohomology of the moduli of parabolic vector bundles, \emph{Proc. Indian
Acad. Sci, (Math. Sci.)} \textbf{95} (1986), 61--77.

\bibitem[NR]{NR} M.S. Narasimhan and S. Ramanan, Deformations of the moduli space of
vector bundles over an algebraic curve, \emph{Ann. Math.} \textbf{101} (1975), 391--417.

\bibitem[Seb]{sebastian}
R. Sebastian, Torelli theorems for moduli of logarithmic
connections and parabolic bundles, \emph{Manuscr. Math.} \textbf{136} (2011), 249--271.

\bibitem[Ses]{seshadri}
C.S. Seshadri, \emph{Fibr\'es vectorieles sur les courbes alg\'ebriques},
\emph{Asterisque} 96, 1982.

\bibitem[Si]{simpson}
C. Simpson, Harmonic bundles on non compact curves, \emph{Jour. Amer. Math. Soc.}
\textbf{3} (1990), 713--770.

\bibitem[Tj]{T}
A.N. Tjurin, An analogue of the Torelli theorem for two-dimensional bundles over
an algebraic curve of arbitrary genus, \emph{Izv. Akad. Nauk SSSR Ser. Mat.}
\textbf{33} (1969), 1149–1170.

\bibitem[We]{We}
A. Weil, Zum beweis des Torelli satzes,
\emph{Nachr. Akad. Wiss. G\"ottingen Math.-Phys. Kl. II} 
2 (1957) 33--53.
 
\bibitem[Yo]{y}
K. Yokogawa, Infinitesimal deformation of parabolic Higgs sheaves, \emph{Inter. Jour.
Math.} \textbf{6} (1995), 125--148.

\end{thebibliography}
\end{document}